\documentclass[draft]{article}

\usepackage{amssymb,amsfonts,amstext}
\usepackage{latexsym}
\usepackage{amsmath}
\newcommand{\NR}{{\mathbb R}}

\newcommand{\NC}{{\mathbb C}}

\newcommand{\NZ}{{\mathbb Z}}

\newcommand{\gr}[2]{\mathop{{\bf #1}({#2})}\nolimits}

\newcommand{\IS}[2]{\sum\limits_{#1}^{#2}}
\newcommand{\ts}{\textstyle}
\newcommand{\DS}[2]{\ts\bigoplus\limits_{#1}^{#2}}

\newcommand{\be}{\begin{equation}}
\newcommand{\ee}{\end{equation}}
\newcommand{\oz}{\overline{z}}
\renewcommand{\Re}{\mathop{\rm Re}\nolimits}
\newcommand{\Ref}[1]{{\rm (\ref{#1})}}
\long\def\comment#1{}

\newcommand{\CH}{{\cal H}}

\newcommand{\CP}{{\cal P}}

\newcommand{\CZ}{{\cal Z}}

\mathchardef\za="710B  
\mathchardef\zb="710C  
\mathchardef\zg="710D  
\mathchardef\zd="710E  
\mathchardef\ze="710F 
\mathchardef\zz="7110  
\mathchardef\zh="7111  
\mathchardef\zy="7112 
\mathchardef\zi="7113  
\mathchardef\zk="7114  
\mathchardef\zl="7115  
\mathchardef\zm="7116  
\mathchardef\zn="7117  
\mathchardef\zx="7118  
\mathchardef\zp="7119  
\mathchardef\zr="711A  
\mathchardef\zs="711B  
\mathchardef\zt="711C  
\mathchardef\zu="711D  
\mathchardef\zf="711E 
\mathchardef\zq="711F  
\mathchardef\zc="7120  
\mathchardef\zw="7121  
\mathchardef\zve="7122  
\mathchardef\zvy="7123  
\mathchardef\zvo="7124  
\mathchardef\zvr="7125 
\mathchardef\zvs="7126 
\mathchardef\zvf="7127  
\mathchardef\zG="7000  
\mathchardef\zD="7001  
\mathchardef\zY="7002  
\mathchardef\zL="7003  
\mathchardef\zX="7004  
\mathchardef\zP="7005  
\mathchardef\zS="7006  
\mathchardef\zU="7007  
\mathchardef\zF="7008  
\mathchardef\zC="7009  
\mathchardef\zW="700A  

\title{On spherical expansions of smooth $\gr{SU}{n}$-zonal functions on the unit sphere in $\NC^n$}

\author{{\bf Agata Bezubik\thanks{e-mail: agatab@math.uwb.edu.pl}}\\
\footnotesize{Institute of Mathematics,}\\[-6pt]
\footnotesize{University of Bia\l{}ystok,}\\[-6pt]
\footnotesize{Akademicka 2, 15-267 Bia\l{}ystok, POLAND} 
\and {\bf Aleksander\ Strasburger\thanks{Corresponding author; e-mail: aleksander\_strasburger@sggw.pl}}\\
\footnotesize{Department of Applied Mathematics, }\\[-6pt]
\footnotesize{Warsaw University of Life Sciences, } \\[-6pt]
\footnotesize{Nowoursynowska 166, 02-787 Warszawa, POLAND}
}

\newtheorem{theo}{Proposition} 
 
\newtheorem{lemma}[theo]{Lemma}
\newtheorem{Theo}{Theorem}
\newtheorem{coro}{Corollary}[theo]
\newtheorem{remark}{Remark}

\newcommand{\qed}{$\Box$} 
\newenvironment{proof}{\begin{trivlist}\item[\hskip\labelsep{\bf Proof.}]}
{\hfill\qed\end{trivlist}}

\date{}
\begin{document}

\maketitle

\thispagestyle{empty}

\begin{abstract} 
\noindent 

We give a self-contained presentation of a novel approach to a construction of spherical harmonic expansions on the unit sphere in $\NC^n$. We derive a new formula for coefficients of the expansion of a smooth zonal function defined on the unit sphere and apply it in some special cases. The expansion for the Poisson--Szeg\"o kernel for the unit ball in $\NC^n$ obtained by our method coincides with the result obtained originally  by G. Folland, and on the other hand disproves results recently presented in a paper of V.A. Menegatto et al..     

\vspace{.4cm}
\noindent{\bf  AMS Math. Subj. Class. (2000):} 58G35, 35F05, 33A75, 33A45

\vspace{.2cm} 

\end{abstract}

\section{Introduction}
This paper presents a part of results from the doctoral dissertation of Agata Bezubik \cite{AB} concerned with spherical expansions of zonal functions on the unit sphere in the complex space $\NC^n$. The aim was to develop further an approach initiated in the papers \cite{BDS1,BDS2} of constructing spherical expansions of smooth zonal functions in terms of their differential characteristic (e.g. the Taylor coefficients) rather than the integral ones (i.e. integral means against appropriate reproducing kernels). 

The complex case treated here, in contrast to the real case discussed in the earlier paper \cite{BDS2}, goes beyond the classical theory in the sense that it requires the use of an orthogonal system in two variables --- more specifically, of an orthogonal system on the unit disc constructed in terms of the so called disc polynomials obtained from the classical Jacobi polynomials, cf. e.g. \cite{Koo,SD}. However this aspect is interesting for us only so far as being the tool for parametrizing the zonal functions on the unit sphere in $\NC^n$. 

Finally let us note that almost at the same time as the doctoral thesis of A. B. was submitted (June 2010), an article \cite{MOP} of V.A.~Menegatto, A.P.~Peron, and C.P.~Oliveira has been published in this Journal (\textit{Collectanea Mathematica\/}). Acknowledging an inspiration of our earlier paper \cite{BDS2}, the authors formulate an expansion theorem (Theorem 2.3 of their paper) with an almost identical content as the Theorem \ref{exp_f_C} --- this ,,almost" referring to certain inaccuracies regarding the various coefficients throughout the paper. We shall pinpoint these problematic expressions in \cite{MOP} in due course below. 

\section{Preliminaries}
We shall be concerned with problems in analysis on the unit sphere  $S=\{z\in \NC^n \ | \  |z|^2=1\}$ in the complex $n$-space $\NC^n$. The norm $|\,\cdot\,|$ is the usual one, defined as $|z|^2=\sum_{j=1}^{n}|z_j|^2$, and derived from the hermitian inner product $(z|w)= \sum_{j=1}^{n}z_j\bar{w_j}$. Given  $z=(z_1,\,\ldots,\,z_n)\in\NC^n$,  we set $
z_j=x_j+iy_j$, with real $x_j,\,y_j$ and $i=\sqrt{-1}$, for $j=1,\,\ldots,\,n$. This gives us a standard identification of $\NC^n$ with the real Cartesian space $\NR^{2n}$ by $(z_1,\,\ldots,\,z_n)\leftrightarrow (x_1,\,\ldots,\,x_n,\,y_1,\,\ldots,\,y_n)$.

The special unitary group $\gr{SU}{n}$ acts transitively on $S$ with the isotropy group of a point $\zh\in S$ isomorphic to the group $\gr{SU}{n-1}$. In fact, taking $\zh=e_n$, the unit vector corresponding to the $n$-th coordinate, we easily see that the isotropy group consists of matrices of the form $\left(\begin{smallmatrix} A&0\\0 &1\end{smallmatrix}\right)$, with the $(n-1)$-square block $A$ in the upper left corner belonging to $\gr{SU}{n-1}$. 

Thus, along with the usual representation of $S$ as the homogeneous space $S=\gr{SO}{2n}/\gr{SO}{2n-1}$ of the orthogonal group $\gr{SO}{2n}$, we can identify the sphere $S$ with the homogeneous space $\gr{SU}{n}/ \gr{SU}{n-1}$. The choice of reference point $\zh\in S$ is not relevant, since isotropy groups are conjugate and it will be  of some advantage not to fix this point by any specific choice. The sphere $S$ is equiped with the normalized Euclidean surface measure $d\zs_{2n-1}=d\zs$ (i.e. $\int_S d\zs=1$), which is $\gr{O}{2n}$-invariant, as well as $\gr{SU}{n}$-invariant.   

By means of the above identification we can express objects attached to the Euclidean space $\NR^{2n}$ in terms of the complex coordinates. For example 
the Euclidean Laplacian $\zD $ on $\NR^{2n}$ is written in the complex form  
\begin{equation}\label{eqn:complex_laplacian}
\zD =
\; 4 \sum\limits_{j=1}^{n}{\frac{\partial^2}{\partial z_j\partial \bar{z}_j}}.
\end{equation}
using the standard notation 
$\frac{\partial}{\partial z_j}=\frac12\left(\frac{\partial}{\partial x_j}-i\frac{\partial}{\partial y_j}\right),\ 
\frac{\partial}{\partial\bar{ z_j}}=\frac12\left(\frac{\partial}{\partial x_j}+i\frac{\partial}{\partial y_j}\right)
$ 
for $j=1,\,\ldots,\,n$.

Likewise, polynomials on $\NR^{2n}$ can be written as polynomials of the complex coordinates $z_j$ and their conjugates $\bar{z_j}$, what will be indicated by writting $P(x,y)=P(z,\,\bar{z})$ for a typical element of the polynomial algebra $\CP(\NR^{2n})$ also denoted as $\CP(\NC^{n})$. For example, the square of the norm  is $
 |z|^2\;= \;\sum_{j=1}^{n}  x_j^2+\sum_{j=1}^{n} y_j^2\; =\; 
 \sum_{j=1}^{n}z_j\bar{z_j}.
$

A departure point for analysis on the unit sphere $S^{d-1}$ is the fundamental decomposition 
\begin{equation}\label{eqn:space-decomposition} 
L^2(S^{d-1},\,d\,\sigma)= {\ts \bigoplus\limits_{l=0}^{\infty}} H^{l}.
\end{equation}
where $ H^{l}$ denotes the space of restrictions to the sphere of harmonic (i.e. annihilated by the Laplacian) and homogeneous  of degree $l$ polynomials on $\NR^d$. A key property of those spaces is their irreducibility under the action of the orthogonal group $\gr{SO}{d}$.  

For the case of the unit sphere in $\NC^n$, the main role is assigned to the group $\gr{SU}{n}$, what requires a refinement of the decomposition \Ref{eqn:space-decomposition} based on the notion of bihomogeneity of polynomials. It is introduced as follows. Given a pair of non-negative integers $p,\,q$ one says that a polynomial $P(z,\,\bar{z})$ on $\NC^n$ is bihomogeneous of degree $(p,\,q)$ if 
\begin{equation}\label{eqn:bihomogen}
 P(\zl z,\bar{\zl}\bar{ z})=\zl^p\; \bar{\zl}^q\; P(z,\bar{z}), \qquad \zl\in \NC, \quad z\in \NC^n.
\end{equation}
The space of bihomogeneous of degree $(p,q)$ polynomials  on $\NC^n$ is denoted by $ \CP^{(p,q)}=\CP^{(p,q)}(\NC^n)$. We notice that  $\CP^{(p,0)}$ contains only holomorphic polynomials, while $\CP^{(0,q)}$ ---  antiholomorphic ones. 
Clearly, $(p,q)$ bihomogeneous polynomials are homogeneous of degree $p+q$ in the real sense and there is a direct sum decomposition 
\[
\CP^l(\NC^{n}) = \DS{p+q=l}{}\CP^{(p,q)}(\NC^n)
\]
of the space of homogeneous of degree $l$ polynomials on $\NC^n$. 
We shall write $\CH^{(p,q)}=\CH^{(p,q)}(\NC^n)=\ker \zD \cap \CP^{(p,q)}(\NC^n)$ for the space of harmonic and bihomogeneous polynomials of bidegree $(p,q)$ (with $\NC^n$ as their domain) --- termed "solid harmonics", and $H^{(p,q)}$ for the space of their restrictions to $S$ --- "surface harmonics". These spaces are irreducible under the action of $\gr{SU}{n}$, their dimensions are given by (c.f. \cite{Koo,SD})
\[
\dim \CH^{(p,q)} =\dim H^{(p,q)} = \frac{(n+p+q-1)(n-2+p)!(n-2+q)!}{p!\ q!\ (n-1)!\ (n-2)!},
\]
and there is an orthogonal decomposition (w.r. to the $L^2(S)$ inner product)
\begin{gather}\label{harm_decomp_by_bideg} 
H^l = \DS{p+q=l}{}H^{(p,q)},
\intertext{and consequently a $\gr{SU}{n}$-irreducible, orthogonal decomposition}
\label{harm_decomp_by_bideg2} 
 L^2(S,\,d\,\sigma)= {\ts \bigoplus\limits_{p,q=0}^{\infty}} H^{(p,q)}.
\end{gather}
Analogous to the real case, there is a decomposition of bihomogeneous polynomials into harmonic components. In fact, $\zD: \CP^{(p,q)}(\NC^n)\to \CP^{(p-1,q-1)}(\NC^n) $ is surjective with the kernel $\CH^{(p,q)}(\NC^n)$ and consequently there is a $\gr{SU}{n}$-invariant decomposition 
\be\label{decomp_by_bideg}
\CP^{(p,q)}(\NC^n) =  {\ts \bigoplus\limits_{k=0}^{\min(p,q)}} r^{2k} \CH^{(p-k,q-k)}(\NC^n).
\ee
The explicit formulae for this decomposition can be deduced from their counterparts in the real case, which were stated in our paper \cite[Theorem 1]{BDS2}.  We just formulate the end result below. It should be noted, however, that the decomposition has a long history, as it was studied independently by several authors, see e.g. \cite{IS,VSz,Koo}.

\begin{Theo}[The canonical decomposition of bihomogeneous polynomials]
\label{can-decomp_C} 
Given $p,q\in \NZ_+$ we set $m=\min(p,\,q)$ and for $0\le k\le m$ we define 
\begin{align}
\zb^{(p,q)}_{\,k}(j) & = 
(-1)^j\frac {(n-1+p+q-2k)(n-2+p+q-2k-j)!}{4^{k+j}\; k!\; j!\; (n-1+p+q-k)!}. \nonumber 
\intertext{and for a polynomial $P\in\CP^{(p,q)}$ we set} 
\label{eqn:can-decomp3_C} 
h^{(p,\,q)}_k (P)& = 
 \sum^{m-k}_{j=0}  \zb^{(p,q)}_{\,k}(j) 
r^{2j} \Delta^{k+j}(P).  
\intertext{%
Then $h^{(p,\,q)}_k(P)\in \CH^{(p-k,q-k)}$ and $P$ can be decomposed into harmonic components}
\label{rozkl_P}
P & = \IS{k=0}{m}r^{2k}h^{(p,\,q)}_k(P).    
\end{align}
The resulting direct sum decomposition 
$\CP^{(p,q)} =  
\bigoplus
_{k=0}^{m} r^{2k} \CH^{(p-k,q-k})$
is orthogonal with respect to the inner product induced by restriction of the inner product in $L^2(S)$ and the maps $ \CP^{(p,q)}\ni P \rightarrow r^{2k}h^{(p,\,q)}_k(P)$ are orthogonal projections corresponding to this decomposition and commuting with the action of $\gr{SU}{n}$. 
\end{Theo}

\section{Decomposition of zonal functions}
\subsection{Preliminaries}
Recall that the isotropy group $\gr{SU}{n}_{\zh}$ of any given point $\zh\in S =S^{2n-1}\subset \NC^n$ is conjugate to $\gr{SU}{n-1}$. A function $f$ defined on $S$ which is invariant under the action of the isotropy group $\gr{SU}{n}_{\zh}$ will be called a zonal function with the pole at $\zh$ (more properly perhaps it should be called ``zonal with respect to $\gr{SU}{n}$''). By a slight abuse of notation we shall call a bihomogeneous polynomial on $\NC^n$ zonal, if its restriction to the unit sphere is zonal. Since the space $\gr{SU}{n}_{\zh}\backslash S$ of orbits of the action of $\gr{SU}{n}_{\zh}$ on the sphere $S$ is naturally identified with the closed disc $D=\{w\in \NC : |w|\leq 1\}$ in the complex plane by factoring the map 
\be\label{eqn:quotient}
S\ni \zx \mapsto (\zx\mid \zh)\in D
\ee
with respect to the natural projection $S \to \gr{SU}{n}_{\zh}\backslash S$, it follows that any zonal function $f$ on $S$ can be uniquely represented in the form 
\be\label{profil_C}
f(\zx)=\zvf((\zx\mid\zh)), \qquad \zx\in S. 
\ee
by a function $\zvf$ defined on the disc $D$, called the profile of $f$.

Further, the (normalized) measure $d\,\zs_{2n-1} $ on the sphere  $S$ is factorized with respect to the map \Ref{eqn:quotient} as follows  
\begin{equation}\label{2.6}
\int_{S^{2n-1}}\!\!f(\zx)d\,\zs_{2n-1}(\zx)\!=\! \frac{n-1}{\zp}\!\int_D(1-|w|^2)^{n-2}d\,\zl(w)\!\int_{S^{2n-3}}\!\!
f(w\zh+\sqrt{1-|w|^2}\zr)d\,\zs_{2n-3}(\zr),
\end{equation}           
where $d\,\zl(w)$ is the Lebesgue measure on the disc $D$ and $S^{2n-3}$ is the unit sphere in the space perpendicular to $\zh$. 

To describe zonal functions we need to introduce a convenient notation and  recall some classical notions. We write $\za = n-2$, since in this form the dimension of $\NC^n$ enters most of the formulas in the sequel. The hypergeometric function is defined by the sum of the series, convergent for all $|t|<1$,  
\be\label{hypergeom} {}_2F_1\left(\begin{matrix}
a, \; b \\
c  \end{matrix};\ t\right) = \IS{j=0}{\infty}\frac{(a)_j(b)_j}{(c)_j}\frac{t^j}{j!}, 
\ee
where $a,\,b,\,c$ are parameters (for us always real),  
and  $(r)_j$ denotes the Pochhammer symbol defined recursively by setting $(r)_0=1$ and  
$(r)_j=r(r+1)\cdots(r+j-1)$, for $j>0$,
for any real (or complex) $r$ and non-negative integer $j$. 
In the following we shall be concerned with the case when the nominator parametres $a,\,b$ are non-positive integers and $c< \min(a,\,b)$, in which case the hypergeometric series  \Ref{hypergeom} terminates and represents a polynomial of degree $\min(-a,\,-b)$. 

The next result provides at the same time explicit expressions for zonal harmonics and the detailed form of the harmonic decomposition of certain bihomogeneous polynomials.  

\begin{theo}Let $p,\,q\in\NZ_+$ be given and $\za=n-2$ as defined above, fix $\zh\in S$ and consider the zonal polynomial $P(z,\oz)=(z \mid\zh)^p \overline{(z\mid \zh)}^q\in\CP^{(p,\,q)}$. The harmonic decomposition \Ref{rozkl_P} of $P$ has the form  
\begin{align}\label{rozw_w_f_hip}
(z\mid \zh)^p\overline{(z\mid \zh)}^q & =  \IS{k=0}{m} |z|^{2k}
 \frac{ p!\;q!\;(\za+1+p+q-2k)!}{k!\;(p-k)!\;(q-k)!\;(\za+1+p+q-k)!} 
\comment{|z|^{p+q} \IS{k=0}{m}(\zx\mid \zh)^{p-k}\overline{(\zx\mid \zh)}^{q-k}\frac{ p!\;q!\;(\za+1+p+q-2k)!}{k!\;(p-k)!\;(q-k)!\;(\za+1+p+q-k)!}}
\nonumber  \\ 
& \times (z\mid \zh)^{p-k}\overline{(z\mid \zh)}^{q-k} {}_2F_1\left(\begin{matrix}-p+k,\; -q+k \\
-\za-p-q+2k \end{matrix}; \ \frac{1}{|(\zx\mid \zh)|^{2}}\right).
\end{align}
\end{theo}

The proof is a straightforward computation based on the formula \Ref{eqn:can-decomp3_C} combined with some chasing of coefficients and can be omitted.

To note is the representation theoretic interpretation of the decomposition \Ref{rozw_w_f_hip}. 
Since the projection maps $P\mapsto h^{(p,\,q)}_k(P)$ commute with the action of $\gr{SU}{n}$, the harmonic polynomials   
\[
(z\mid \zh)^{p-k}\overline{(z\mid \zh)}^{q-k} {}_2F_1\left(\begin{matrix}-p+k,\; -q+k \\
-\za-p-q+2k \end{matrix}; \ \frac{1}{|(\zx\mid \zh)|^{2}}\right)
\]
in the above decomposition are zonal elements of $\gr{SU}{n}$-irreducible spaces $\CH^{(p-k,\,q-k)}$. Since the latter are known to contain one dimensional subspace of zonal elements, we may state the following.
\begin{coro}Given arbitrary nonnegative integers $k,\,j$, any zonal polynomial in $\CH^{(k,\,j)}$ with pole at $\zh$ is proportional to the following  
\be\label{solid_zonal}
\CZ^{(k,\,j)}_{\zh}(z,\,\oz) = (z\mid \zh)^{k}\overline{(z\mid \zh)}^{j} {}_2F_1\left(\begin{matrix}-k,\; -j \\ -\za-k-j \end{matrix}; \ \frac{1}{|(\zx\mid \zh)|^{2}}\right).
\ee
\end{coro}
It is customary to normalize the zonal polynomials by requiring that they assume the value $1$ at the pole. Since the Chu-Vandermonde formula (cf. \cite[Cor.~2.2.3, p.~67]{AA}) gives  
\[
{}_2F_1\left(\begin{matrix}-k,\; -j \\ -\za-k-j \end{matrix}; \ 1\right)= \frac{(\za+1)_k(\za+1)_j}{(\za+1)_{k+j}},
\]
we set 
\be \label{reprod_kernel} 
Z^{(k,\,j)}_\zh(\zx)=\frac{(\za+1)_{k+j}}{(\za+1)_k(\za+1)_j}(\zx\mid \zh)^{k}\overline{(\zx\mid \zh)}^{j} {}_2F_1\left(\begin{matrix}-k,\; -j \\ -\za-k-j \end{matrix}; \ \frac{1}{|(\zx\mid \zh)|^{2}}\right),
\ee
obtaining so called \emph{reproducing kernels} for spaces $H^{(k,\,j)}$. They are characterized by the identity 
\be\label{eqn:reproducing_identity}
\dim H^{(k,\,j)} \int_Sf(\zx)Z^{(k,\,j)}_\zh(\zx) d\,\zs(\zx)= f(\zh),\qquad f\in H^{(k,\,j)}.
\ee
An equivalent expression for zonal kernels uses Jacobi polynomials $P^{(\zm,\,\zn)}_{m}(t)$, cf. \cite[Section 2.4.3]{DXu},  
\be\label{zonalny-C}
Z_{\zh}^{(k,\,j)}(\zx)=(\zx \mid \zh)^{k-m}\overline{(\zx \mid \zh)}^{j-m} \; \frac{P^{(\za,|k-j|)}_{m}(2|(\zx \mid \zh)|^2-1)}{P^{(\za,|k-j|)}_{m}(1)}, \qquad \zx\in S,
\ee
where $m=\min(k,\,j)$. Taking account of the normalization used, we can rewrite the decomposition \Ref{rozw_w_f_hip} for surface harmonics as follows  
\[
(\zx\mid\zh)^p\overline{(\zx\mid\zh)}^q  =\IS{k=0}{m}\zg^{p,\,q}_k Z_{\zh}^{(p-k,\,q-k)}(\zx),
\]
with
\be \label{wspolczynnik-gamma}
\zg^{p,\,q}_k =
\frac{p!q!(\za+1+p+q-2k)(\za+p-k)!(\za+q-k)!}{k!\za!(\za+1+p+q-k)!(p-k)!(q-k)!}.
\ee

Profile functions of the reproducing kernels $Z_{\zh}^{(p,q)}(\zx)$ are precisely the disc  polynomials in the normalization given in the book of Dunkl and Xu \cite[Section 2.4.3]{DXu}. 
For $p,\,q\in \NZ_+$ and $\za$ as above we set 
\be\label{Dunkl_wiel}
W^\za_{p,\,q}(w)=\frac{(\za+1)_{p+q}}{(\za+1)_p(\za+1)_q}w^p\bar w^{q}{}_2F_1\left(\begin{matrix}-p,\; -q\\ 
-\za-p-q \end{matrix} ;\ \frac{1}{|w|^{2}}\right), \qquad w\in D,\  
\ee
so that $ Z_{\zh}^{(p,\,q)}(\zx)= W^\za_{p,\,q}((\zh\mid\zx))$.

In virtue of surjectivity of the map \Ref{eqn:quotient} the decomposition formula 
\Ref{rozw_w_f_hip} implies the following decomposition of basic monomials on the disc.

\begin{coro}For any pair of nonnegative integers $p,\,q$, the following is valid
\be\label{w^pw^q}
w^p\bar w^q =\IS{k=0}{m}\zg^{p,\,q}_k\;W^\za_{p-k,\,q-k}(w), \qquad 
w\in D, 
\ee
where $\zg^{p,\,q}_k $ are given by \Ref{wspolczynnik-gamma}.
\end{coro}

Substitution of $w=1$ in \Ref{w^pw^q} together with the normalization $W^\za_{(p,\,q)}(w)=1$ gives the following ``decomposition of unity''
\be\label{jedynka}
1=\displaystyle{\IS{k=0}{m}\frac{p!q!(n-1+p+q-2k)(n-2+p-k)!(n-2+q-k)!}{k!(n-2)!(n-1+p+q-k)!(p-k)!(q-k)!}}.
\ee
which is crucial for the proof of our main result, the Expansion Theorem \ref{exp_f_C} in the next Section.

\begin{remark}
The expansion \Ref{w^pw^q} coincides, with the necessary adjustment of notation, with the one given in the paper of W\"unsche \cite[formula (3.10), p. 142]{Wunsche}, where it is approached from a different angle. However, in the paper of Menegatto and al., \cite[Lemma 2.1, p.~153]{MOP} W\"unsche's  formula is incorrectly quoted, with an extra factorial appearing in the expression for their coefficients  $c^k_{n,p,q}$ (replacing our $\zg^{p,\,q}_k$). As a consequence, their decomposition of unity (2.5) on page 154 does not agree with our formula \Ref{jedynka}, and is not correct, as the following simple example shows. Taking $m=q=2, n=3$ in their formula (2.5), one gets  
$$\IS{k=0}{\min(2,3)}\frac{2!3!(2-k+2-2)!(3-k+2-2)!(2+3-2k+2-1)!}{(2-2)!k!(2-k)!(3-k)!(2+3-k+2-1)!}=14.9\neq 1.$$ 

\end{remark}

\subsection{A general expansion theorem}

We now come to the main result of this paper, which generalizes Theorem 2 of our  earlier paper \cite{BDS2} to the present context of spheres in complex $n$-space.  

Let $f$ be a zonal function on the unit sphere $S\subset\NC^n$. Recalling this means that $f$ is invariant with respect to the isotropy group $\gr{SU}{n}_{\zh} $ of a point $\zh\in S$, it is easy to infer from the group invariance of the decomposition \Ref{harm_decomp_by_bideg2}, that its spherical harmonic expansion is of the form
\[
f(\zx)=\IS{p,q=0}{\infty}d_{p,\,q}\;\dim H^{(p,\,q)}\; Z_{\zh}^{(p,\,q)}(\zx), 
\] 
where $d_{p,\,q}=d_{p,\,q}(f)$ are scalar coefficients. Now, regarding the issue of computing these coefficients, the general theory of orthogonal expansions together with the identity \Ref{eqn:reproducing_identity} implies that 
\be\label{wsp_d_p,q}
d_{p,\,q}(f)=\int_S f(\zx)Z_{\zh}^{(p,\,q)}(\zx)\;d\sigma (\zx).
\ee 
By virtue of \Ref{2.6} this reduces to the integral of the profile $\zvf$ of $f$,
\[
d_{p,\,q}(f)= \frac{\za+1}{\pi}\int_D \zvf(w)W^\za_{p,\,q}(w)(1-|w|^2)^\za\;d\zl(w).
\] 
 
Similarly to the case of the real sphere, discussed in our earlier paper \cite{BDS2}, for sufficiently regular zonal functions the integral formula for the expansion coefficients can be replaced by a differential formula involving Taylor coefficients of the profile function.

To simplify notation we shall write $\partial, \bar\partial$ instead of $\frac{\partial}{\partial w},\; \frac{\partial}{\partial \bar w}.$

\begin{Theo}\label{exp_f_C}
Let $f$ be a zonal function with pole at $\zh \in S$, $\zvf$, $f(\zx)=\zvf((\zx\,|\,\zh))$. Assume the  profile function $\zvf$ is real analytic on the interior of the disc $D$ and its Taylor series around $0$ 
\[
\IS{j,\,k=0}{\infty}\frac{\partial^j \bar\partial^k\;\zvf(0)}{j!k!}  w^j\bar w^k 
\]
is absolutely convergent on the unit circle.
Then the coefficients $d_{p,\,q}$ of the spherical harmonic expansion of $f$, 
\be \label{complex_expansion}
f(\zx)=\IS{p,q=0}{\infty}d_{p,\,q}\;\dim H^{(p,\,q)}\; Z_{\zh}^{(p,\,q)}(\zx), 
\ee
are given by the formulae 
\be\label{wspol_rozniczk}
d_{p,\,q}=d_{p,q}(\zvf) = (n-1)!
\IS{k=0}{\infty} \dfrac{\partial^{p+k} \bar\partial^{q+k}\zvf(0)}{k!(n-1+p+q+k)!},  
\ee
and the expansion is absolutely and uniformly convergent on $S$.  
\end{Theo}

\begin{proof}
The assumption on the Taylor series of $\zvf$ implies that the series converges uniformly to $\zvf$ on the closed disc $D$, hence we can write  
$$
 \zvf(w)=\IS{j,k=0}{\infty}\frac{\partial^{j} \bar\partial^{k}\zvf(0)}{j!k!}   w^j\bar w^k .
$$
Using the formula \Ref{w^pw^q} we get 
$$\zvf(w)=\IS{j,k=0}{\infty}\frac{\partial^{j} \bar\partial^{k}\zvf(0)}{j!k!} \IS{l=0}{m}\zg^{j,\,k}_l W^\za_{j-l,\,k-l}(w).$$
Substituting $w= (\zx\mid\zh)$ turns the above into 
\be\label{eqn:sum_to_order}
f(\zx)=\IS{j,k=0}{\infty}\frac{\partial^{j} \bar\partial^{k}\zvf(0)}{j!k!} \IS{l=0}{m}\zg^{j,\,k}_l  Z_{\zh}^{(j-l,\,k-l)}(\zx),\qquad m=\min(j,\,k). 
\ee
To justify the change of the summation order it is enough to show that the inner sum is uniformly bounded (w.r. to $j,\,k$). Since  $$\bigl|Z_{\zh}^{(j-l,\,k-l)}(\zx)\bigr|=\bigl|W^\za_{j-l,\,k-l}((\zx\mid\zh)\bigr|\leq 1,$$
cf. \cite[\textbf{2.4.3}, Property $(iii)$, p. 57]{DXu}, we get
$$\left|\IS{l=0}{m}\zg^{j,\,k}_l W^\za_{j-l,\,k-l}(w)\right|\leq \IS{l=0}{m}\zg^{j,\,k}_l =1, $$
using \Ref{jedynka} at the end. 

Now in the expansion \Ref{eqn:sum_to_order} we first sum  the terms proportional to the zonal harmonic $Z_{\zh}^{(p,\,q)}(\zx)$ of a given bidegree. The resulting coefficient is 
\[
\IS{l=0}{\infty}\left(\frac{\partial^{p+l} \bar\partial^{q+l}\zvf(0)}{(p+l)!(q+l)!} \zg^{p+l,\,q+l}_l  \right)Z_{\zh}^{(p,\,q)}(\zx)
\]  
Using \Ref{wspolczynnik-gamma} we get
\begin{gather*}
\frac{1}{(p+l)!(q+l)!} \zg^{p+l,\,q+l}_l   
= \frac{(\za+1+p+q)(\za+p)!(\za+q)!} {p!q!l!\za!(\za+1+p+q+l)!} \\
=\dim H^{(p,\,q)} \frac{(\za+1)!}{l!(\za+1+p+q+l)!}.
\end{gather*}
This finishes the proof. 
\end{proof}

\begin{remark}
This result was obtained by one of the authors [A.B.] of the present paper in the course of the work on her doctoral dissertation, which was submitted to the Warsaw University of Technology in June 2010. An analogous result is given in the paper \cite{MOP} of Menegatto and al., however the proof given there rests on the use of their decomposition of unity (2.5) on page 154, which is not correct as it stands. 
\end{remark}

\subsection{A variant of the Funck--Hecke formula}

The expansion formulae (\ref{complex_expansion}-\ref{wsp_d_p,q}) immediately imply the following   form of the classical Funk--Hecke theorem for complex spheres, cf. \cite[Theo.~4.4]{EQ}. 

\begin{coro}
If $f$ is a zonal function $f(\zx)=\zvf((\zx\,|\,\zh))$ satisfying the assumptions of the Theorem \ref{exp_f_C} above, then for every complex spherical harmonic $Y^{(p,\,q)}\in H^{(p,\,q)} $, 
\begin{gather}\label{cmplx_FH1}
\int_S Y^{(p,\,q)}(\zx)f(\zx)\;d\sigma (\zx) =
d_{p,\,q}(\zvf )Y^{(p,\,q)}(\zh), 
\intertext{where} 
\label{cmplx_FH2}
d_{p,\,q}(\zvf)   = (n-1)!
\IS{k=0}{\infty}\displaystyle{\frac{\partial^{p+k} \bar\partial^{q+k}\zvf(0)}{k!(n-1+p+q+k)!}}.
\end{gather}
\end{coro}
\begin{proof}
In fact, substituting the expansion formula \Ref{complex_expansion} into the integral on the left-hand-side of equation \Ref{cmplx_FH1} and changing the order of integration and summation we get
\[
\int_S Y^{(p,\,q)}(\zx)f(\zx)\;d\sigma (\zx)  = \IS{p,q=0}{\infty}d_{p,\,q}\;\dim H^{(p,\,q)}\;\int_S Y^{(p,\,q)}(\zx) Z_{\zh}^{(p,\,q)}(\zx)\;d\sigma (\zx) 
\]
By virtue of \Ref{eqn:reproducing_identity} and the orthogonality of spherical harmonic of different bihomogeneity we get the result.
\end{proof}

Comparing with the formula (4.3) of Quinto, \cite[p.~258]{EQ}, see also \Ref{2.6} above, we see that 
\[
\int_D W^{(p,\,q)}(w)\zvf(w)(1-|w|^2)^{n-2}\; d\zl(w)={\pi} (n-2)!
\IS{k=0}{\infty}\dfrac{\partial^{p+k} \bar\partial^{q+k}\zvf(0)}{k!(n-1+p+q+k)!}.
\]
This is reminiscent of the classical Pizzetti's formula, cf. \cite{pa1}, however this connection will be considered elsewhere.
 
\subsection{Applications}

In the following we give two applications of the main theorem.  

\subsubsection{Plane wave expansion}
\begin{coro}
Given $\zh\in S$ let $x\mapsto e^{i\Re (x \mid \zh)}$ be the plane wave in $\NC^n$ with normal $\zh$.    Then writing  $x=|x|\zx$ with $\zx\in S$, one has the complex spherical expansion 
\be\label{wykl_Re}
e^{i\Re (x \mid \zh)}=(n-1)!\Bigl(\frac{|x|}{2}\Bigl)^{-n+1}\IS{p,q=0}{\infty}\,i^{p+q}\dim \CH^{(p,q)} J_{p+q+n-1}(|x|)\; Z_{\zh}^{(p,q)}(\zx),
\ee
where the Bessel functions $J_\zn(r)$ of the first kind and order $\zn$ are given by 
\be\label{Bessel}
J_\zn(r)= \Bigl(\frac{r}{2}\Bigr)^\zn\IS{k=0}{\infty}
\frac{(-1)^k}{\zG(k+1)\zG(k+\zn+1)}\Bigl(\frac{r}{2}\Bigr)^{2k}.
\ee
The expansion \Ref{wykl_Re} is absolutely and uniformly convergent on every ball in $\NC^n$.  
\end{coro}
\begin{remark}
The Bessel functions, which are coefficients in this expansion, depend only on the total degree of homogeneity $p+q$. Therefore the expansion \Ref{wykl_Re} can be reduced to the usual plane wave expansion in $\NR^{2n}$, cf. eg. \cite{far,BDS1}
\[
e^{i\Re (x \mid \zh)}=(n-1)!\Bigl(\frac{|x|}{2}\Bigl)^{-n+1}\IS{l=0}{\infty}\,i^{l}\dim \CH^{l}J_{l+n-1}(|x|)\; Z_{\zh}^{l}(\zx)
\] 
by using the summation formula for the zonal kernels
\[
\dim \CH^l \; Z^l_{\zh}(\xi)=\IS{p+q=l}{}\dim\CH^{(p,q)}\; Z^{(p,q)}_{\zh }(\xi).
\]
 
\end{remark}
\begin{proof}
By differentiation of the profile function $\zvf (w)=e^{ir \Re w}$, where $r=|x|$, we obtain for the coefficients of the plane wave expansion the following expressions 
$$\begin{array}{llll}
d_{p,q}(\zvf)&=&\displaystyle{(n-1)!\IS{k=0}{\infty}\frac{(-1)^k \; i^{p+q}\; r^{p+q+2k}}{2^{p+q+2k}k! (n-1+p+q+k)!}}\\
&=&\displaystyle{(n-1)!\; i^{p+q}\Bigl(\frac{r}{2}\Bigl)^{-n+1}\Bigl(\frac{r}{2}\Bigl)^{p+q+n-1}\;\IS{k=0}{\infty}\frac{(-1)^k}{k! (n-1+p+q+k)!}\Bigl(\frac{r}{2}\Bigl)^{2k}}.\\
\end{array}$$
By comparison with the formula \Ref{Bessel} the result follows. 
\end{proof}

\subsubsection{Expansion of Poisson--Szeg\"o kernel}

Let us remind that the Poisson--Szeg\"o kernel is defined by the formula 
\be\label{j_P-Sz}
P_n(z,\,\zh)=\frac{(1-|z|^2)^n}{|1-(z|\zh)|^{2n}}, \qquad (z,\,\zh)\in B\times S,
\ee 
where $B\subset \NC^n$ is the unit ball and the unit sphere $S$ is its boundary. It plays the same role with respect to the Laplace---Beltrami operator associated to the Bergman metric on $B$ (cf. eg. \cite{Rud}) as the usual Poisson kernel in relation to the Euclidean Laplacian in the real case, i.e. the formula 
$$
u(z)= \int_S P_n(z,\zh)f(\zh)\,d\sigma,\qquad z\in B,  
$$
expresses functions in the unit ball annihilated by the Laplace--Beltrami operator in terms of their boundary values on the sphere $S$.  

In the paper \cite{Fo2} Folland derived the spherical harmonic expansion of the Poisson--Szeg\"o kernel and gave the explicit expressions for the coefficients in terms of hyper\-geometric functions. We show below that Folland's expansion follows from our expansion Theorem \ref{exp_f_C}, which applies here in view of the fact that the function 
$\zx \mapsto P_n(r\zx,\zh) =\dfrac{(1-r^2)^n}{|1-r(\zx|\zh)|^{2n}}$
is a zonal function on $S$ with the pole at $\zh$.

\begin{Theo}\label{wn_j-P-C}
For $r\in[0,1)$ and $\zh\in S$ the spherical harmonic expansion of the Poisson--Szeg\"o kernel   $\zx\mapsto P_n(r\zx,\zh)$ has the form

\be\label{j_Poissona-Sz}
P_n(r\zx,\zh)=\frac{(1-r^2)^n}{|1-r(\zx|\zh)|^{2n}} =\IS{p,q=0}{\infty}\dim \CH^{(p,q)} S_n^{p,q}(r) Z_{\zh}^{(p,q)}(\zx),
\ee
where 
\be\label{wspol_S}
S_n^{p,q}(r)=r^{p+q}\;\frac{(p+n-1)!(q+n-1)!}{(n-1)!(p-1)!(q-1)!}\IS{k=0}{\infty}\frac{(p+k-1)!(q+k-1)!}{(n-1+p+q+k)!}\frac{r^{2k}}{k!}
\ee
or equivalently 
\be
S_n^{p,q}(r) = r^{p+q} {}_2F_1\left(\begin{matrix}p,\;q\\ p+q+n\end{matrix};r^2\right){\bigg /} {}_2F_1\left(\begin{matrix}p,\;q\\ p+q+n\end{matrix};1\right) .
\ee
\end{Theo}

\begin{remark}
Folland in \cite{Fo2} approaches the problem of this expansion along different lines. He first establishes a special case of the solution of the Dirichlet problem corresponding to a bihomogeneous spherical harmonic as the boundary value --- at this stage already the series $S_n^{p,q}(r)$ for the coefficients enter, and from that derives the expansion of the kernel  by examining the convergence of the right hand side of \Ref{j_Poissona-Sz}. 
\end{remark}

\comment{
Our approach is straightforward and based on the use of the expansion Theorem \ref{exp_f_C} applied to the denominator in the formula the binomial formula  for the nominator in \Ref{j_P-Sz} together with  other factor  of the formula. The summation of the resulting series uses the Pfaff--Saalsch\"utz identity, cf. \cite[Theorem 2.2.6]{AA}. 
}

\begin{proof}  
From the expansion given as formula (2.9) on p.~6 in  \cite{Koo},  cf. also \cite{SD}
$$
\dfrac{1}{|1-w|^{2n}}
= \dfrac{1}{(1-w)^n(1- \overline{w})^n }= \IS{p,q=0}{\infty}\binom{p+n-1}{n-1}\cdot \binom{q+n-1}{n-1} w^p\overline{w}^q,
$$ 
setting $\zvf_r(w)=  |1-rw|^{-2n}$ we immediately see that   
\[
 \partial^{j}\bar\partial^{k}\zvf_r(0) = 
r^{j+k}(n)_j(n)_k . 
\]
Now applying formula \Ref{wspol_rozniczk} of the Theorem \ref{exp_f_C} to the function $\zvf_r((\zx|\zh))$ (with $\zh\in S$ kept fixed) we see that the following holds 
\begin{lemma}\label{rozw_f_gen}
For any $\zx, \zh\in S$  
\begin{gather}
\dfrac{1}{|1-r(\zx|\zh)|^{2n}}= \nonumber \\ 
\IS{p,q=0}{\infty}\dim \CH^{(p,q)}Z_{\zh}^{(p,q)}(\zx)r^{p+q}
\IS{k=0}{\infty}\frac{(n-1+p+k)!(n-1+q+k)!}{(n-1)!(n-1+p+q+k)!}\frac{r^{2k}}{k!}.
\end{gather}
\end{lemma}
To finish the proof of Theorem \ref{wn_j-P-C} we need to expand(\footnote{There is no attempt to make this expansion in the paper of Menegatto et al.\cite{MOP}}) 
$$
(1-r^2)^n\, r^{p+q} 
\IS{k=0}{\infty}\frac{(n-1+p+k)!(n-1+q+k)!}{(n-1)!(n-1+p+q+k)!}\frac{r^{2k}}{k!}   
$$
in terms of powers of $r^2$ and show that the expansion coincides with the formula \Ref{wspol_S}.  

Using the binomial formula we get  
\begin{align}\label{eqn:binom}
&\hphantom{={}\;}(1-r^2)^n\,r^{p+q}
\IS{k=0}{\infty}\frac{(n-1+p+k)!(n-1+q+k)!}{(n-1)! (n-1+p+q+k)!}\frac{r^{2k}}{k!}\nonumber  \\
&= \dfrac{r^{p+q}}{(n-1)!}\IS{j=0}{n}(-1)^j\binom{n}{j}r^{2j}\IS{k=0}{\infty}\frac{(n-1+p+k)!(n-1+q+k)!}{ (n-1+p+q+k)!}\frac{r^{2k}}{k!} \nonumber \\
&= \dfrac{r^{p+q}}{(n-1)!}  \IS{k=0}{\infty}r^{2k}\IS{j=0}{n}(-1)^j\binom{n}{j}\frac{(n-1+p+k-j)!(n-1+q+k-j)!}{ (n-1+p+q+k-j)!(k-j)!}
\end{align}

The inner sum in the last equality, after rewriting it in the form of a hypergeometric sum, can be computed by using the Pfaff--Saalsch\"utz identity, cf. \cite[Theorem 2.2.6]{AA},
$$ 
{}_3F_2\left(\begin{matrix}  
-n,\;a, \; b \\
\noalign{\smallskip}
c,\; 1+a+b-n-c & \end{matrix};\  1\right)=\IS{j=0}{n}\frac{(-n)_j (a)_j(b)_j}{(c)_j (1+a+b-c-n)_j}\frac{1}{j!} =
\frac{(c-a)_n(c-b)_n}{(c)_n(c-a-b)_n}.  
$$  

In fact, recalling the identity satisfied by the Pochhammer symbol  
$(-1)^j\dfrac{t!}{(t-j)!}=(-t)_j$,\ for\ $t>j$, 
we see that the inner sum in the expression \Ref{eqn:binom} can be rewritten as a hypergeometric type sum 
\be\label{eqn:sum_hyper}
\frac{(n-1+p+k)!(n-1+q+k)!}{k!(n-1+p+q+k)!}\IS{j=0}{n}\frac{(-n)_j(-k)_j(1-p-q-k-n)_j}{(1-p-k-n)_j(1-q-k-n)_j}\frac{1}{j!}.
\ee 
Thus setting $a=-k$, $b=1-p-q-k-n$, $c=1-p-k-n$ in the Pfaff--Saalsch\"utz identity and using  $(t)_n=(-1)^n(1-t-n)_n$, the sum \Ref{eqn:sum_hyper} is seen to to be equal  to 
\[
\frac{(p+n-1)!(q+n-1)!(p+k-1)!(q+k-1)!}{k!(p-1)!(q-1)!(p+q+n+k-1)!}. 
\]
Substituting this expression into \Ref{eqn:binom} and comparing with \Ref{wspol_S} concludes the proof of the Theorem \ref{wn_j-P-C}.
\end{proof}

\begin{remark}
The expansion of the Poisson--Szeg\"o kernel given in the Corollary 3.3 of \cite{MOP} differs  at two important points from the expansion \Ref{wspol_S}, which coincides with the one given originally by Folland in \cite{Fo2}.  Firstly, on the right hand side of the given formula the factor $(1-r^2)^q$ is not expanded, but more importantly, the inner sum there is just a numerical factor, because the factor  $r^{2j}$ is missing in all its terms. The consequence of this omission is that the formula (3.11) in the Theorem 3.4 of the paper \cite{MOP} expresses $S_n^{p,q}(r)$ as a polynomial, what is definitely not true.  

\end{remark}

\begin{footnotesize}

\end{footnotesize}

\end{document}